\documentclass{amsart}
\usepackage{amsmath}
\usepackage{amsthm}
\usepackage{amssymb}
\usepackage{amsbsy}
\usepackage{amsfonts}
\usepackage{amstext}
\usepackage{amscd}
\usepackage{tikz}
\usepackage{graphicx}

\numberwithin{equation}{section}
\theoremstyle{plain}
\newtheorem{thm}{Theorem}[section]
\newtheorem{prop}[thm]{Proposition}
\newtheorem{cor}[thm]{Corollary}
\newtheorem{lem}[thm]{Lemma}
\theoremstyle{definition}
\newtheorem{exa}[thm]{Example}

\newtheorem{rem}[thm]{Remark}
\newtheorem{defi}[thm]{Definition}

\newcommand{\real}{\mathbb{R}}

\begin{document}
\title{Convergence of the Fourth Moment and Infinite Divisibility}

\author{Octavio Arizmendi}

\address{Universit\"{a}t des Saarlandes, FR $6.1-$Mathematik. 66123 Saarbr\"{u}cken, Germany }
\email{arizmendi@math.uni-sb.de}
\thanks{Supported by DFG-Deutsche Forschungsgemeinschaft 
Project SP419/8-1.}
\date{\today}

\maketitle
\begin{abstract}
In this note we prove that, for infinitely divisible laws,  convergence of the fourth moment to 3 is sufficient to ensure convergence in law to the Gaussian distribution. Our results include infinitely divisible measures with respect to classical, free, Boolean and monotone convolution. A similar criterion is proved for compound Poissons with jump distribution supported on a finite number of atoms. In particular, this generalizes recent results of Nourdin and Poly. 
\end{abstract}

\section{Introduction and Statement of Results}

  In a seminal paper,  Nualart and Peccati \cite{NuPe} proved a convergence criterion for multiple integrals in a fixed chaos with respect to the classical
Brownian motion to the standard normal distribution $\mathcal{N}(0,1)$ which gives a drastic simplification for the so-called method of moments. More precisely, let $(W_t)_{t \geq0}$ be a standard Brownian motion. For
every square-integrable function $f$ on $\mathbb{R}^m_+$ we denote by $I^W_m(f)$ the $m$-th
multiple Wiener-It\^o stochastic integral of $f$ with respect to $W$. 

\begin{thm}[\cite{NuPe}] \label{T1} 
Let $\{X_n=I^W_m(f_n)\}_{n>0}$ be a sequence of multiple Wiener-It\^ o integrals in a fixed $m$-chaos with $ E[X_n^2] \rightarrow 1$ and denote $\mu_{X_n}$ the distribution of $X_n$. Then the following are equivalent
\begin {enumerate}
\item $ E[X_n^4] \rightarrow 3$
\item $\mu_{X_n}\rightarrow \mathcal{N}(0,1)$
\end{enumerate}

\end{thm}

In free probability, the standard semicircle distribution  $\mathcal{S}(0,1)$  plays the role of the gaussian distribution. Recently, it was proved by Kemp et al. \cite{KNPS} that the Nualart-Peccati criterion also holds for
the free Brownian motion $(S_t)_{t\geq0}$ and its multiple Wigner integrals $I_m^S(f)$. 

\begin{thm}[\cite{KNPS}] \label{T2}
Let $\{X_n:=I_m^S(f_n)\}_{n>0}$ be a sequence of  multiple Wigner integrals in a fixed $m$-chaos with $ E[X_n^2] \rightarrow 1$ denote $\mu_{X_n}$ the distribution of $X_n$.  Then the following are equivalent 

\begin{enumerate}
\item $ E[X_n^4] \rightarrow 2$
 \item $\mu_{X_n}\rightarrow \mathcal{S}(0,1)$
\end{enumerate}
\end{thm}

In this paper we prove analogous results to Theorem \ref{T1} and  Theorem \ref{T2} in the setting of infinitely divisible laws.  Let $ID(*)$ and $ID(\boxplus)$
denote the classes of probability measures which are infinitely divisible with respect to classical convolution  $*$  and free convolution $\boxplus$, respectively. 

\begin{thm}\label{T3} Let $\{\mu_n=\mu_{X_n}\}_{n>0}$ be a sequence of probability measures with variance $1$ and mean $0$ such that $\mu_n\in ID(*)$. If $ E[X_n^4] \rightarrow 3$
then $\mu_{X_n}\rightarrow \mathcal{N}(0,1)$.
\end{thm}
\begin{thm}\label{T4} Let $\{\mu_n=\mu_{X_n}\}_{n>0}$ be a sequence of probability measures with variance $1$ and mean $0$ such that $\mu_n\in ID(\boxplus)$. If $ E[X_n^4] \rightarrow 2$
then $\mu_{X_n}\rightarrow \mathcal{S}(0,1)$. 
\end{thm}

To complete the picture we show that the monotone probability of Muraki also fits in our framework, namely, Theorems \ref{T3} and \ref{T4} are also true in the monotone case.
\begin{thm}\label{T7} Let $\{\mu_n=\mu_{X_n}\}_{n>0}$ be a sequence of $\rhd$-infinitely divisible probability measures with common variance $1$ and mean $0$ and denote by $\mathcal{A}(0,1)$ the arcsine distribution with  mean zero and variance $1$. If $ E[X_n^4] \rightarrow 1.5$
then $\mu_{X_n}\rightarrow \mathcal{A}(0,1)$.
\end{thm}
Moreover, we can extend our results to compound Poisson distributions whose L\'evy measure has finite support. We only state the free version for the sake of clarity.
\begin{thm}\label{poisson} Let $\mu_{X_n}\in ID(\boxplus)$ be random variables and denote by $\pi_\boxplus(\lambda,\nu)$ the free compound poisson measure with  rate $\lambda$ and jump distribution $\nu:=\sum_{i=1}^k a_i\delta_{b_i}.$  If $E[X_n^k]\rightarrow m_k(\pi(\lambda))$  for 
$ i=1,\dots,2k+2$  then $\mu_{X_n}\rightarrow \pi_\boxplus(\lambda,\nu).$
\end{thm}

We want to emphasize that our approach relies on a third notion of non commutative independence (Boolean independence) and the so-called Bercovici-Pata bijections, $\Lambda$, $\Lambda^\rhd$ and $\mathbb{B}$  
(see Section 2). This  gives another example on how this third notion of independence sometimes regarded as uninteresting because of its simplicity can provide a better understanding in other notions of independence.

Some natural questions arise from the theorems above. What is the relation between Theorems \ref{T1} and \ref{T2}, and Theorems \ref{T3} and \ref{T4}?
Multiple integrals are in general not infinitely divisible\footnote{A counterexample for the free case is given by the third Chebychev polynomial of a semicircle,
i.e., $x= s^3-2s$.}. 
However, this is true in the first or second chaos.
In particular, from Theorem \ref{poisson} we may recover Theorem 4.3 in Nourdin and Poly \cite{NoPo}.  Another interesting question coming from Theorem \ref{T7} is if Theorem \ref{T1} is also valid for multiple integrals with respect to monotone Brownian motion; to the knowledge of the author this is still an open question.

\section{Preliminaries}

\subsection{The Cauchy Transfom }

We denote by $\mathcal{M}$ the set of Borel probability measures on $\real$. The upper half-plane and the lower half-plane are respectively denoted as $\mathbb{C}^+$ and $\mathbb{C}^-$. 

Let $G_\mu(z) = \int_{\real}\frac{\mu(dx)}{z-x}$ $(z \in \mathbb{C}^+)$ be the Cauchy transform of $\mu \in \mathcal{M}$.

The relation between weak convergence and the Cauchy Transform is the following (see e.g. \cite{BaiSil}).

\begin{prop}
\label{WCCT}Let $\mu _{1}$ and $\mu _{2}$ be two probability measures on $
\mathbb{R}$ and  
\begin{equation*}
d(\mu _{1},\mu _{2})=\sup \left\{ \left\vert G_{\mu _{1}}(z)-G_{\mu
_{2}}(z)\right\vert ;\Im(z)\geq 1\right\} .
\end{equation*}
Then $d$ is a distance which defines a metric for the weak topology of probability
measures. In particular,  $G_\mu(z$) is bounded in $\{z:\Im(z)\geq 1\}$.\end{prop}
In other words, a sequence of probability measures $\left\{ \mu
_{n}\right\} _{n\geq 1}$ on $\mathbb{R}$ converges weakly to a probability
measure $\mu $ on $\mathbb{R}$ if and only if for all $z$ with $\Im(z)\geq 1$ we have
\begin{equation*}
\lim_{n\rightarrow \infty }G_{\mu _{n}}(z)=G_{\mu }(z).
\end{equation*}

\subsection{The Jacobi Parameters}

Let $\mu$ be a probability measure with all the moments. The Jacobi parameters  $\gamma _{m}=\gamma _{m}(\mu )\geq 0,\beta
_{m}=\beta _{m}(\mu )\in\mathbb{R}$, are defined by the recursion 
\begin{equation*}
xP_{m}(x)=P_{m+1}(x)+\beta _{m}P_{m}(x)+\gamma _{m-1}P_{m-1}(x),
\end{equation*}
where the polynomials  $P_{-1}(x)=0,$ $P_{0}(x)=1$ and $(P_{m})_{m\geq 0}$ is a sequence of orthogonal monic polynomials with respect to $\mu $, that is,
\begin{equation*}
\int_{\mathbb{R}}P_{m}(x)P_{n}(x)\mu (dx)=0\text{ \  \  \  \ if }m\neq n.
\end{equation*}
A measure $\mu$ is supported on $m$ points iff $\gamma_{m-1} = 0$ and $\gamma_n> 0$ for $n = 0,\dots , m-2$.

The Cauchy transform may be expressed as a continued fraction in terms of the Jacobi parameters, as follows.
\begin{equation*}
G_{\mu }(z)=\int_{-\infty }^{\infty }\frac{1}{z-t}\mu (dt)=\frac{1}{z-\beta
_{0}-\dfrac{\gamma _{0}}{z-\beta _{1}-\dfrac{\gamma _{1}}{z-\beta _{2}-\cdots}}}
\end{equation*}

In the case when $\mu$ has $2n+2$-moments we can still make an orthogonalization procedure until the level $n$. In this case the Cauchy transform has the form

\begin{equation}\label{expansionjacobi}
G_{\mu }(z)=\frac{1}{z-\beta
_{0}-\dfrac{\gamma _{0}}{z-\beta _{1}-\dfrac{\gamma _{1}}{~~~\dfrac{\ddots}{ z-\beta _{n}-
\gamma _{n}G_{\nu }(z)}}}}
\end{equation}
where $\nu$ is a probability measure.

\subsection{Different notions of convolution}

In non-commutative probability, there exist various notions of independence. In this paper we will focus on the notions of independence coming from universal products as classified by Muraki \cite{Mu2}: tensor(classical), free, Boolean and monotone independence.
\subsubsection{Classical convolution}
Recall that the classical convolution of two probability measures $\mu_1,\mu_2$ on $\mathbb{R}$ is defined as the probability measure $\mu_1*\mu_2$ on $\mathbb{R}$ such that
$C_{\mu_1*\mu_2}(t)=C_{\mu_1}(t)+C_{\mu_2}(t),t \in \real,$
where $C_{\mu}(t) = \log \hat{\mu}(t)$ with $\hat{\mu}(t)$ the characteristic function of $\mu$. Classical convolution corresponds to the sum of tensor independent random variables: $\mu_a*\mu_b=\mu_{a+b}$, for $a$ and $b$ independent random variables. 
The (classical) cumulants are the coefficients $c_n = c_n (\mu)$
in the series expansion $$C_{\mu}(t) =\sum^\infty_{n=1}\frac{c_n}{n!} t^n.$$
   
Similar convolutions and related transforms exist for the free, Boolean and monotone theories.

\subsubsection{Free convolution}

Free convolution was defined in \cite{Voi85} for compactly supported probability measures and later extended in \cite{Maa} for the case of finite variance, and in \cite{BeVo93} for the general unbounded case.
Let $G_\mu(z)$ be the Cauchy transform of $\mu \in \mathcal{M}$ and let
$F_\mu(z)$ be its reciprocal $\frac{1}{G_\mu(z)}$. It was proved in Bercovici and Voiculescu \cite{BeVo93}  that there are positive numbers $\eta$ and $M$ such that $F_\mu$ has a right inverse $F_\mu^{-1}$ defined on the region
$\Gamma_{\eta,M}:= \{z\in\mathbb{C}^+; |Re(z)| < \eta Im(z),~~|z|>M\}.$

The Voiculescu transform of $\mu$ is defined by $\phi _{\mu }\left( z\right) =F_{\mu }^{-1}(z)-z,$
on any region of the form $\Gamma_{\eta,M}$ where $F^{-1}_{\mu}$ is defined.

The \emph{free additive convolution} of two probability measures $\mu_1,\mu_2$ on $\mathbb{R}$ is the
probability measure $\mu_1\boxplus\mu_2$ on $\mathbb{R}$ such that 
$$\phi_{\mu_1\boxplus\mu_2}(z) = \phi_{\mu_1}(z) + \phi_{\mu_2}(z), \quad \text{for } z\in\Gamma_{\eta_1,M_1}\cap \Gamma_{\eta_2,M_2}.$$ Free additive convolution corresponds to the sum of free random variables: $\mu_a\boxplus\mu_b=\mu_{a+b}$, for $a$ and $b$ free random variables.  The free cumulants \cite{sp1} are the coefficients
$\kappa_n= \kappa_n (\mu)$ in the series expansion 
\begin{equation}\label{free cumulants}
\phi_\mu(z) = \sum_{n=1}^\infty \kappa_n z^{1-n}.
\end{equation}

\subsubsection{Boolean convolution}

The \emph{Boolean convolution} \cite{S-W} of two probability measures $\mu_1 , \mu_2$ on $\mathbb{R}$ is defined as
the probability measure $\mu_1\uplus\mu_2$ on $\mathbb{R}$ such the transform $K_\mu(z) = z- F_\mu (z)$,
(usually called \emph{self-energy}), satisfies
$$K_{\mu_1\uplus\mu_2}(z) = K_{\mu_1}(z) + K_{\mu_2}(z) ,~~~z\in\mathbb{C}^+.$$ 
Boolean convolution corresponds to the sum of Boolean-independent random variables. Boolean cumulants are defined as 
the coefficients $r_n = r_n (\mu)$ in the series
\begin{equation}\label{Boolean cumulants}
K_\mu(z) =\sum^\infty_{n=1}r_nz^{1-n}.
\end{equation}

\subsubsection{Monotone convolution}
The monotone convolution was defined in \cite{Mu1} and extended to unbounded measures in \cite{Franz}. The \emph{monotone convolution} of two probability measures $\mu_1 , \mu_2$ on $\mathbb{R}$ is defined
as the probability measure $\mu_1\rhd\mu_2$ on $\mathbb{R}$ such that
$$F_{\mu_1\rhd\mu_2}(z) = F_{\mu_1}(F_{\mu_2}(z)) ,~~~z\in\mathbb{C}^+.$$ 
Monotone convolution corresponds to the sum of monotone  independent random variables.
Recently, Hasebe and Saigo \cite{HaSa} have defined a notion of monotone cumulants
$(h_n )_{n\geq1}$ which satisfy that $h_n (\mu^{\rhd k}) = kh_n (\mu) $.

\subsubsection{Moment-Cumulant formulae}
For a measure $\mu$ its classical, free, Boolean and monotone \textit{cumulants} $(c_n)_{n\geq1}$, $(k_n)_{n\geq1}$, $(r_n)_{n\geq1}$, $(h_n)_{n\geq1}$, satisfy the moment-cumulant formulas
\begin{equation}\label{MCF}
m_n
= \sum_{\pi\in \mathcal{P}(n)}c_{\pi}^a
=\sum_{\pi\in NC(n)}\kappa_{\pi}^a
= \sum_{\pi\in \mathcal{I}(n)}r_{\pi}^a
= \sum_{(\pi,\lambda)\in \mathcal{M}(n)}\frac{h_{\pi}^a}{|\pi|!},
\end{equation}
where, for a sequence of complex numbers $(f_n)_{n\geq 1}$ and a partition $\pi=\{V_1,\dots,V_i\}$, we define $f_{\pi}:=f_{|V_1|}\cdots f_{|V_i|}$ and $|\pi|$ is the number of blocks of the partition $\pi$ and where $\mathcal{P}(n), NC(n), \mathcal{I}(n),\mathcal{M}(n)$ denote the set of all, non-crossing, interval and monotone partitions (see \cite{HaSa,sp1,S-W}), respectively. We note here that convergence of the first $n$ moments is equivalent to the convergence of the first $n$ cumulants.

\subsection{Infinite divisibility}

\begin{defi} Let $\circledast$ be one of the above convolutions, namely, $\circledast\in \{*,\uplus,\boxplus, \triangleright\}$. A probability measure $\mu$ is said to be $\circledast$-infinitely divisible if for each $n\in\mathbb{N}$ there exist $\mu_n \in \mathcal{M}$ such that $\mu=\mu_n^{\circledast n}.$ We will denote by $ID(\circledast)$ the set of $\circledast$-infinitely divisible measures. 
\end{defi}

Recall that a probability measure $\mu $ is infinitely divisible in the
classical sense if and only if its classical cumulant transform $\log 
\widehat{\mu }$ has the L\'{e}vy-Khintchine representation
\begin{equation}
\log \widehat{\mu }(u)=i\gamma u-\int_{\mathbb{R}}(e^{iut}-1-\frac{iut}{1+t^2} )\frac{1+t^2}{t^2}\sigma \left( dt\right) ,
\text{ \ \ }u\in\mathbb{R},  \label{levykintchine clasica}
\end{equation} where $\gamma\in \mathbb{R}$ and   $\sigma$ is a finite measure on $\real$.
 If this representation exists, the pair $(\gamma,\sigma )$ is determined in a unique way and is
called the (classical) generating pair of $\mu $. In this case we denote $\mu$ by $\rho_*^{\gamma,\sigma}$

From the Voiculescu Transform one has a representation analogous to L\'evy-Kintchine\'{}s.  Bercovici and Voiculescu 
\cite{BeVo93} proved that a probability measure $\mu $ is $\boxplus $-infinitely divisible if and only if there exists a finite measure $\sigma $ on $
\mathbb{R}$ and a real constant $\gamma $\ such that
\begin{equation}
\phi _{\mu }(z)=\gamma +\int_{\mathbb{R}}\frac{1+tz}{z-t}\sigma (dt),\qquad z\in 
\mathbb{C}^{+}.
\end{equation}
The pair $(\gamma ,\sigma )$ is called the \emph{free} generating pair of $\mu $ and we denote $\mu$ by $\rho_\boxplus^{\gamma,\sigma}.$

For the Boolean case, things are easier. As shown by Speicher and Wourodi \cite{S-W}, any probability measure is infinitely divisible with respect to the Boolean convolution and there is also
a Boolean L\'evy-Kintchine representation. Indeed, it follows then by general Nevanlinna-Pick theory that for any probability measure $\mu$ there exists a real constant $\gamma$ and a finite measure
  $\sigma$ on $\real$, such that
\begin{equation}
K_\mu(z)=\gamma+\int_{\real}\frac{1+tz}{z-t}\,\sigma(d t), \qquad z\in{\mathbb C}^+.
\label{eq2}
\end{equation}
The pair $(\gamma ,\sigma )$ is called the \emph{Boolean} generating pair of $\mu $ and we denote $\mu$ by $\rho_\uplus^{\gamma,\sigma}.$

A characterization of $\triangleright$-infinitely divisible measures was done by Muraki \cite{Mu3} and  Belinschi \cite{Be}. A probability measure $\mu$ belongs to $ID(\rhd)$  if and only if there exists a composition semigroup of reciprocal Cauchy transforms $F_{s+t}=F_s\circ F_t=F_t\circ F_s$ and $F_1=F_\mu$. In this case the map $t\mapsto F_t(z)$ is differentiable for each fixed $z$ in $\mathbb {R}$ and we define the mapping $A_\mu$ on
${\mathbb C}^+$ by 
\[
A_\mu(z)=\frac{d}{d t}\Big|_{t=0}F_t(z), \qquad z\in{\mathbb C}^+.
\]
For mappings of this form there exists $\gamma\in\mathbb{R}$ and a finite
measure $\sigma$, such that
\begin{equation}\label{monotone L-K}
A_\mu(z)=-\gamma-\int_{{\mathbb R}}\frac{1+tz}{z-t}\,\sigma(d t).
\end{equation}
This is the L\'evy-Khintchine formula for monotone convolution and in this case we denote $\mu$ by $\rho_\rhd^{\gamma,\sigma}.$
The monotone cumulants $h_n$ are the coefficients in the series
\begin{equation}\label{monotone cumulants}
-A_\mu(z) =\sum^\infty_{n=1}h_nz^{1-n}.
\end{equation}

An important class of infinitely divisible measures is the class of compound Poisson distributions since any infinitely divisible measure can be approximated by them. 
\begin{defi}
Let $\circledast \in \{*,\uplus,\boxplus, \triangleright\}$. Denote by $k^\circledast_n$ the cumulants with respect to the convolution $\circledast$.  If
$k^\circledast_{n}(\mu)=\lambda m_{n}(\nu )$
for some $\lambda >0$ and some distribution $\nu $ we say that $\mu$ is a $\circledast$-compound
Poisson distribution with rate $\lambda $ and jump distribution $\nu$  and denote $\mu$ by $\pi_\circledast(\lambda,\nu)$. If $\lambda=1$ and $\nu =\delta_1$ then $k^\circledast_{n}(\mu)=1$, for all $n\in\mathbb{N}$, and we call $\mu$ simply a $\circledast$-Poisson.
\end{defi}

\subsection{Bercovici-Pata bijections}

From the various L\'evy-Kintchine representations it is readily seen that there is a bijective correspondence between the infinite divisible measures with respect to the different notions of independence. These bijections are called the Bercovici-Pata bijections, since they were studied by Bercovici and Pata \cite{BePa} in relation to limit theorems and domain of attractions.
\begin{defi}\label{BPB}
\begin{enumerate}

\item The (classical-to-free) Bercovici-Pata bijection $\Lambda:ID(*)\to ID(\boxplus)$ is defined by the application
$\rho^{\gamma,\sigma}_*\mapsto \rho^{\gamma,\sigma}_\boxplus$.
\item The (Boolean-to-free) Bercovici-Pata bijection $\mathbb{B}:\mathcal{M}\to ID(\boxplus)$ is defined by the application
$\rho^{\gamma,\sigma}_\uplus\mapsto \rho^{\gamma,\sigma}_\boxplus$.
\item The (classical-to-monotone) Bercovici-Pata bijection $\Lambda^\rhd:ID(*)\to ID(\rhd)$ is defined by the application
$\rho^{\gamma,\sigma}_*\mapsto \rho^{\gamma,\sigma}_\rhd$.
\end{enumerate}
\end{defi}

The weak continuity of $\Lambda$ and $\Lambda^{-1}$ was proved in \cite{BNT}. On the other hand, the weak continuity of $\mathbb{B}$ and $\mathbb{B}^{-1}$ follows from the continuity of the free and Boolean convolution powers since $\mathbb{B}(\mu)=(\mu^{\boxplus 2})^{\uplus 1/2}$. Finally the weak continuity of $\Lambda^\rhd$ was proved in Hasebe \cite{Ha}. In summary, the arrows of the following commutative diagram are weakly continuous.
$$\begin{tikzpicture}[node distance=1.3cm, auto]
  \node (boo) {$ID(\uplus)=\mathcal{M}$};
  \node (cla) [below of=boo] {$ID(\ast)$};
  \node (fre1) [left of=cla] {$ $};
  \node (fre) [left of=fre1] {$ID(\boxplus)$};
  \node (mon1) [right of=cla] {$ $};
  \node (mon) [right of=mon1] {$ID(\triangleright)$};
  
  \draw[->] (boo) to node [swap] {$\mathbb{B}$} (fre);
  \draw[->] (boo) to node {$ $} (cla);
  \draw[->] (boo) to node {$ $} (mon);
    \draw[->] (fre) to node {$\Lambda^{-1}$} (cla);
  \draw[->] (cla) to node {$\Lambda^{\triangleright}$} (mon);
\end{tikzpicture}$$
\begin{rem}\label{rem2}
If follows from (\ref{free cumulants}) and (\ref{Boolean cumulants}) that the Boolean cumulants of $\mu$ are free cumulants of its image under the Boolean Bercovici-Pata bijection $\mathbb{B}$, namely, $r_n(\mu)=k_n( \mathbb{B}(\mu))$. Similarly, 
$c_n(\mu)=k_n( \Lambda(\mu))$ and $c_n(\mu)=h_n( \Lambda^\rhd(\mu))$.
\end{rem}

\section{Convergence to the Gaussian distribution}
For a random variable $X$ with all the moments,  mean 0 and variance 1, let us denote by
$S_n^{\ast}(X)=(X_1+X_2+\cdots+X_n)/\sqrt{n}$
the normalized sum of $n$ independent copies of $X$. The so-called Central Limit Theorem states that $S_n^{\ast}(X)$ converges, as $n\to\infty$,  to the standard Normal distribution $\mathcal{N}(0,1)$.

On the other hand, the free Central Limit Theorem (see \cite{Boz}, \cite{Voi85}) states that 
 the normalized sum of free copies of $X$ 
converges weakly to the 
standard semicircle distribution  $\mathcal{S}(0,1)$,  
with density  $$\frac{1}{2\pi}\sqrt{4-x^2}, \quad x\in[-2,2].$$ 
 
Similarly, the Boolean Central Limit Theorem (\cite{S-W}) states that the normalized sum of \emph{Boolean}-independent copies of $X$ converges weakly to the 
Bernoulli distribution, $\mathbf{b}:=1/2\delta_{-1}+1/2\delta_1$ .

For monotone independence, the limiting distribution for the Central Limit Theorem (\cite{Mu1}) is the Arcsine distribution, with density $$\frac{1}{\pi\sqrt{2-x^2}} \quad x\in[-\sqrt{2},\sqrt{2}].$$ 

The standard proof for convergence to any of these ``gaussian" distributions consists in showing the convergence of \emph{all} the moments. In this section we will prove that, when staying among infinitely divisible laws, convergence of the 4th moment is enough, namely we will prove Theorems \ref{T3}, \ref{T4} and \ref{T7}.

The main observation is that the following simple lemma together with the continuity properties of the Bercovici-Pata bijections gives the desired results.

\begin{lem}\label{lemma1}
 Let $X_n$ be random variables with variance $1$ and mean $0$, if $E(X_n^4)\rightarrow 1$ then $\mu_{X_n}\rightarrow b$ a symmetric Bernoulli  $\mathbf{b}$.
\end{lem}
\begin{proof}Let $Y_n=X^2_n$. Then $E[Y_n]\rightarrow1$ and $E[Y_n^2]\rightarrow1$. This means that $Var(Y_n)\rightarrow0$ and thus $Y_n\rightarrow1$ in $L^2$. By the condition that $E[X_n]=0$ we see that $\mu_{X_n}\rightarrow \mathbf{b}$.
\end{proof}

Now, we can prove Theorems \ref{T3}, \ref{T4} and \ref{T7} which we state again for the convenience of the reader.
\begin{thm}\label{T9} Let $\{\mu_{X_n}\}_{n>0}$ be a sequence of probability measures with variance $1$ and mean 0.\begin{enumerate}
\item If $\mu_{X_n}\in ID(*)$ and $ E[X_n^4] \rightarrow 3$ then $\mu_{X_n}\rightarrow \mathcal{N}(0,1)$.
\item If $\mu_{X_n}\in ID(\boxplus)$ and $ E[X_n^4] \rightarrow 2$ then $\mu_{X_n}\rightarrow \mathcal{S}(0,1)$.
\item If $\mu_{X_n}\in ID(\rhd)$ and $ E[X_n^4] \rightarrow 1.5$ then $\mu_{X_n}\rightarrow \mathcal{A}(0,1)$.
\end{enumerate}
\end{thm}

\begin{proof} We first prove (2). 
 Recall from Definition \ref{BPB} that $\mathbb{B}$ stands for the Boolean-to-free Bercovici-Pata bijection. Assume $\mu_n=\mathbb{B}(\nu_n)$, for some $\nu_n$. Then by Remark \ref{rem2}, $\nu_n$ has variance $1$ and mean 0, and $m_4(\nu_n)\rightarrow 1$. The previous lemma applies and yields that $\nu_n\rightarrow \mathbf{b}$. By the continuity of $\mathbb{B}$, we deduce that $ \mathbb{B}(\nu_n)\rightarrow \mathbb{B}(\mathbf{b})={S}(0,1)$. Parts (1) and (3) follows the same lines by changing  $\mathbb{B}$ by $\mathbb{B}\circ\Lambda^{-1}$  and $\mathbb{B}\circ\Lambda^{-1}\circ\Lambda^{\rhd}$, respectively.
\end{proof}

\begin{exa}[Boolean Central Limit Theorem]
 Let $X_i$ be  Boolean independent identically distributed random variables with $E(X_i)=0$ and $E(X_i^2)=1$. Then, the random variable $Y_n=\frac{X_1+X_2+...X_n}{\sqrt{n}}$ is infinitely divisible. Moreover, $E(Y_n)=0$, $E(Y_n^2)=1$ and $E(Y_n^4)=r_2+r_4/n\rightarrow r_2=1$. Thus, by Lemma \ref{lemma1}, we see that $Y_n\rightarrow \mathbf{b}$. 
\end{exa}

\begin{exa}[Convergence of the Poisson Distribution to the Normal Distribution]
 Let $X_n$ be a random variable with distribution $\pi(n,\delta_1)$, the random variable $Y_n=\frac{X_n-n}{\sqrt{n}}$ converges weakly to $N(0,1)$. Indeed, $Y_n=\frac{X_n-n}{\sqrt{n}}$ is infinitely divisible. Moreover, $E(Y_n)=0$, $E(Y_n^2)=1$ and $E(Y_n^4)=c_4/n+3c_2^2=1/n+3\rightarrow3$. Hence, by Theorem \ref{T9} we see that $Y_n\rightarrow N(0,1)$
\end{exa}

A similar argument proves the following criteria for approximation to the Poisson distributions. This shall be compared to the results in \cite{NoPe}.

\begin{prop} Fix $\circledast\in \{*,\uplus,\boxplus, \triangleright\}$, let $\mu_n\in ID(\circledast)$ be a sequence of probability measures such that $\mu_n\in ID(\circledast)$ for all $n$, and denote by $\pi_\circledast$ the $\circledast$-Poisson measure. If $\kappa_i(X_n)\rightarrow1$ for $ i=1,2,3,4$ then $\mu_{X_n}\rightarrow \pi_\circledast.$

\end{prop}

\begin{proof} The Boolean Poisson $\pi_\uplus$ has distribution $1/2\delta_{0}+1/2\delta_1$. Thus one needs to show that convergence of the first fourth moments is enough to ensure weak convergence, but this is just a simple modification of Lemma \ref{lemma1}. From this point the proof of Theorem \ref{T9} is also valid here by just replacing $\mathbf{b}$ by $1/2\delta_{0}+1/2\delta_1$.
\end{proof}

Since the proofs given above seem to be produced ``ad hoc", one may ask if a more general phenomenan is hidden or this case is very particular. This is the content of next section.

\section{Convergence  to Compound Poisson distributions}

In this section we consider convergence to compound Poisson and more generally convergence to freely infinitely divisible measures with L\'evy measure whose support is finite. 

We first prove the following basic lemma, which replaces Lemma \ref{lemma1}.

\begin{lem}\label{lemma2}
Let $\mu$ be a probability measure with Jacobi parameters $\{\gamma_i,\beta_i\}_{i=1}^k$ with $\gamma_k=0$. If $\{\mu_n\}_{n\in\mathbb{N}}$ is a sequence of probability measures such that $\gamma_i(\mu_n)\rightarrow \gamma_i(\mu)$ and $\beta_i(\mu_n)\rightarrow\beta_i(\mu_n)$ for all $i=0,\dots,k$  then $\mu_n$ converges weakly to $\mu$.
\end{lem}
\begin{proof}
Let $G_\mu$ be the Cauchy transform of $\mu_n $. By Equation (\ref{expansionjacobi}) we can expand $G_\mu$ as a continued fraction as follows 
$$G_{\mu_n }(z)=\frac{1}{z-\beta
_{0}-\dfrac{\gamma _{0}}{~~~\dfrac{\ddots}{ z-\beta _{n}-\gamma _{n}G_{\nu }(z)}}}, $$
where $\nu$ is some probability measure. Now, recall that $G_{\nu }(z)$ is bounded in the set $\{z : \Im(z)\geq 1\}$ and thus, since $\gamma_n\rightarrow0$ we see that $\gamma _{n}G_{\nu }(z)\rightarrow 0 $. This implies the pointwise convergence
$$G_{\mu_n }(z) \rightarrow 
\frac{1}{z-\beta
_{0}-\dfrac{\gamma _{0}}{~~~\dfrac{\ddots}{ z-\beta _{n}}}}$$
 in the set $\{z : \Im(z)\geq 1\}$, which then implies the weakly convergence $\mu_n\to\mu$.
\end{proof}

Now we are in position of proving the main result of the section which contains Theorem \ref{T7} as a special case.  
\begin{thm}\label{poisson2} Let $\circledast\in \{*,\uplus,\boxplus, \triangleright\}$ and $\rho_\circledast^{\gamma, \sigma}\in ID(\circledast)$ with $\circledast$-L\'evy pair $(\gamma,\sigma)$.  Furthermore, assume that $\sigma=\sum_{i=1}^k a_i\delta_{b_i}$.
 If $\{\mu_n\}_{n\in\mathbb{N}}$ is a sequence of $\circledast$-infinitely probability measures such that $m_n(\mu)\rightarrow m_n(\rho_\circledast^{\gamma, \sigma})$,  for $ i=1,\dots,2k+2$ , then $\mu_n$ converges weakly to $ \rho_\circledast^{\gamma, \sigma}$.
\end{thm}

\begin{proof}
 The proof is a mere modification of the proof of Theorem \ref{T3}. The only further observations to be made in this case are the following. First,  if $\rho_\circledast^{\gamma, \sigma}$ is a $\uplus$-infinitely probability measure with $\uplus$-L\'evy pair $(\gamma,\sigma)$ and $\sigma=\sum_{i=1}^k a_i\delta_{b_i}$, then $\rho_\circledast^{\gamma, \sigma}$ is of the form $\sum_{i=1}^{k+1} a_i\delta_{b_i}$. This easily follows from the Boolean L\'evy-Kintchine representation. Second, the convergence of the first $2k+2$ moments is equivalent to the convergence of the first $2k+2$ Jacobi parameters.  Thus  Lemma \ref{lemma2} settles the Boolean case. Finally, we may use again the continuity of the Bercovici-Pata bijection to get the result for the other cases.
\end{proof}

As an example we get  the analog of Theorem 1.1 in Deya and Nourdin \cite{NoDe}, for the so-called Tetilla Law.
\begin{exa} Consider $\mathcal{T}$ the distribution with density \begin{equation*} \label{comw2}
f(t)=\frac{\sqrt{3}}{2\pi \mid t\mid}\left( \frac{3t^2+1}{9h(t)}-h(t)\right),~~~~\mid t\mid\leq\sqrt{(11+5\sqrt{5})/2}, 
\end{equation*}
where 
$$h(t)=\sqrt[3]{\frac{18t^2+1}{27}+\sqrt{\frac{t^2(1+11t^2-t^4)}{27}}}.$$ 
This is the distribution of free commutator $s_1s_2+s_2s_1$ where  $s_1$ and $s_2$ are free semicircle variables. If  $X_n\in ID(\boxplus)$ is a sequence of random variables such that $E[X_n^i]\rightarrow m_i(\mu)$, for $ i=1,...,6,$ then $X_n$ converges in distribution to $\mathcal{T}$. Indeed, $\mathcal{T}$ is a free compound Poisson  $\pi_\boxplus(\lambda,\nu)$ with $\lambda=2$ and $\nu=\mathbf{b}$, thus by Theorem \ref{T7} we get the desired result.
\end{exa}

Finally, as a direct consequence of Theorem \ref{poisson2} we recover a theorem of Nourdin and Poly \cite{NoPo}.

\begin{cor} [Theorem 4.3 \cite{NoPo}] Let $f\in L_s^2(\mathbb{R}^2_+)$ with $0< rank(f) <\infty$  let $\mu_0\in\mathbb{R}$ and let $A\sim S(0,\mu^2_0) $ be
independent of the underlying free Brownian motion $S$. Assume that $|\mu_0|+||f||_{L^2(\mathbb{R}^2_+)}>0$ and set
$$Q(x)=x^{2(1+\mathbf{1}_{\{\mu_0\neq0\}})}\prod^{a(f)}_{i=1}(x-\lambda_i(f))^2.
$$
Let $F_n$ be a sequence of double Wigner integrals. Then as $n\to\infty$, we have

(i)$F_n\rightarrow A + I^S_2(f))$ in law

if and only if the following are satisfied:

(ii-a) $k_2(F_n)\rightarrow k_2(A + I^S_2(f))$;

(ii-b) $\sum^{degQ}_{r=3}\frac{Q^{(r)}(0)}{r!}k_r(F_n)\rightarrow \sum^{degQ}_{r=3}\frac{Q^{(r)}(0)}{r!}( I^S_2(f))$;

(ii-c) $k_r(F_n)$ for $a(f)$ consecutive values of $r$, with $r \geq 2(1+\mathbf{1}_{\{\mu_0\neq0\}}).$
\end{cor}
\begin{proof} The fact the \emph{(i)} implies \emph{(ii)} follows from the boundedness of the sequence. To prove that \emph{(ii)} implies \emph{(i)} we note that since $F_n$ is in the first and second chaos then $F_n$ is freely  infinitely divisible. Moreover $F$ has a L\'evy pair $(\gamma,\sigma)$ with $\sigma$ with finite support. We note here that \emph{(i-a)} corresponds to the Gaussian part.\end{proof}

\end{document}